\documentclass[12pt]{amsart}

\usepackage{pb-diagram}

\usepackage{amsfonts}
\usepackage{amsmath}
\usepackage{amssymb}

\newtheorem{theorem}{Theorem}[section]
\newtheorem{proposition}{Proposition}[section]
\newtheorem{lemma}{Lemma}[section]

\newtheorem{definition}{Definition}[section]

\newtheorem{remark}{Remark}[section]

\numberwithin{equation}{section}

\begin{document}

\title{A uniformization theorem in complex Finsler geometry}

\author[Ningwei Cui, Jinhua Guo,  Linfeng Zhou]{Ningwei Cui$^{*}$, Jinhua Guo$^{**}$,  Linfeng Zhou$^\dagger$}

\address{$^{*}$ Department of Mathematics, Southwest Jiaotong University, Chengdu, 610031, P.R. China}
\email{ningweicui@swjtu.edu.cn}

\address{$^{**}$ School of Mathematics Science, East China Normal University, Shanghai, 200241, P.R. China}
\email{51185500004@stu.ecnu.edu.cn}

\address{$^\dagger$ School of Mathematics Science,  East China Normal University, Shanghai, 200241, P.R. China}
\email {lfzhou@math.ecnu.edu.cn}


\thanks{
The first author is supported by NSFC (No. 11401490), the third author is supported by the National Key Research and Development Program of China under Grant
2018AAA0101001.}

\begin{abstract} 
In complex Finsler geometry, an open problem is: does there exist a weakly K\"ahler Finsler metric which is not K\"ahler?

In this paper, we give an affirmative answer to this open problem. More precisely, we construct a family of the weakly K\"ahler Finsler metrics which are non-K\"ahler. 
The examples belong to the unitary invariant complex Randers metrics. Furthermore, a uniformization theorem of the unitary invariant complex Randers metrics with constant holomorphic curvature is proved under the weakly K\"ahler condition. 

\end{abstract}

\maketitle

\section{Introduction}
In Hermitian geometry, the vanishing of the torsion of the Hermitian connection means that the metric is K\"ahler.  In the complex Finsler geometry, since the torsion of the Chern-Finsler connection has a horizontal part and a mixed part, a complex Finsler metric to be K\"ahler is quite different. There are three kinds of notions which are the strongly K\"ahler, K\"ahler and weakly K\"ahler metric corresponding to the vanishing of the different parts of the torsion of the Chern-Finsler connection \cite{AP}.

However, the relations of the three K\"ahler definition seem rather subtle. Chen and Shen proved that a K\"ahler Finsler metric is actually a strong K\"ahler Finsler metric \cite{CS}. The well-known examples of smooth complex Finsler metrics: Kobayashi and Carath\'eodory metrics, which agree on a bounded strictly convex domain, are weakly K\"ahler metrics \cite{AP}. It is open whether they are K\"ahler or not. In \cite{XZh2}, Xia and Zhong wrote: ``{\em we also do not know whether there exists a weakly K\"ahler Finsler metric which is not a K\"ahler Finsler metric.}" Thus the problem: {\em does there exist a weakly K\"ahler Finsler metric which is not K\"ahler},  remains open. 

In this paper, we give an affirmative answer to this open problem. More precisely, we construct a family of the complex Randers metrics which are weakly K\"ahler but not K\"ahler. Actually the following classification theorem is proved. 
\begin{theorem}
An unitary invariant complex Randers metric $F=\sqrt{r\phi(t,s)}$ defined on a domain $D \subset \mathbb{C}^n$, where $r:=|v|^2$, $t:=|z|^2$, $s:=\frac{|\langle z,v \rangle |^2}{r}$, $z\in D$, and $v\in T_zD$, is weakly K\"ahler if and only if
\[\phi=\Big (\sqrt{f(t)+\frac{tf^{'}(t)-f(t)}{2t}s}+\sqrt{\frac{tf^{'}(t)+f(t)}{2t}s}\Big )^2\]
or
\[\phi=f(t)+f^\prime(t)s\]
where $f(t)$ is a positive smooth function.   
\end{theorem}

\begin{remark} (1) When $\phi=f(t)+f^{\prime}(t)s$, it is easy to see that the metric $F$ is a Hermitian-K\"ahler metric. 

(2) In \cite{Zhong}, Zhong proved that if the unitary invariant complex Finsler metric $F$ is K\"ahler if and only if it is Hermiatian-K\"ahler. 
When $$\phi=\Big (\sqrt{f(t)+\frac{tf^{'}(t)-f(t)}{2t}s}+\sqrt{\frac{tf^{'}(t)+f(t)}{2t}s}\Big )^2,$$ obviously, the complex Finsler metric $F=\sqrt{r\phi(t,s)}$ is not Hermitian,  and thus $F$ is weakly K\"ahler but not K\"ahler. 
\end{remark}

As we know, the uniformization theorem in Hermitian geometry tells us that a complete, simple-connected K\"ahler manifolds of constant holomorphic (sectional) curvature is isometric to one of the standard models: the Fubini-Study metric on $\mathbb{C}P^n$, the flat metric on $\mathbb{C}^n$ and the Bergmann metric on the unit ball in $\mathbb{C}^n$  \cite{Tian}. In complex Finsler geometry, one natural problem is to classify the weakly K\"ahler or K\"ahler Finsler metric with constant holomorphic curvature.  

Roughly speaking, this problem is too ambitious to solve if one does not impose any condition on the metrics. In \cite{XZh1}, Xia and Zhong  classified the unitary invariant weakly complex Berwald metrics of constant holomorphic curvature.
However, if we assume the metrics are the unitary invariant Randers type, we can obtain a complete classification theorem which is similar to the uniformization theorem in Hermitian geometry. In this paper, we proved the following uniformization theorem.

\begin{theorem} \label{mt}
Let $F=\sqrt{r\phi(t,s)}$ be an unitary invariant complex Randers metric defined on a domain $D\subset \mathbb{C}^n$, where $r:=|v|^2$, $t:=|z|^2$, $s:=\frac{|\langle z,v \rangle |^2}{r}$. Assume $F$ is not a Hermitian metric. If $F$ is a weakly K\"ahler Finsler metric and has a constant holomorphic curvature $k$ if and only if 
\begin{enumerate}
\item $k=4$, $\phi=(\sqrt{\frac{t}{c^2+t^2}-\frac{t^2}{(c^2+t^2)^2}s}+\sqrt{\frac{c^2}{(c^2+t^2)^2}s})^2$ defined on $D=\mathbb{C}^n\setminus \{0\}$;
\item $k=0$, $\phi=c(\sqrt{t}+\sqrt{s})^2$ defined on $D=\mathbb{C}^n$;
\item $k=-4$, $\phi=\sqrt{\frac{t}{c^2-t^2}+\frac{t^2}{(c^2-t^2)^2}s}+\sqrt{\frac{c^2}{(c^2-t^2)^2}s}$ defined on $D=\{z: |z|<\sqrt{c}\}$
\end{enumerate}
where $c$ is a positive constant.
\end{theorem}

\begin{remark} It is hopeful to generalize above theorem to the more general complex Randers metrics. 

\end{remark}

The rest of this article is organized as follows. In section 2, we recall some notations and formulas in complex Finsler geometry. In section 3, we introduce $U$ and $W$ to simplify
the  weakly K\"ahler equation of the unitary invariant complex Finsler metrics. We also obtain the formula of the holomorphic curvature of the unitary invariant complex Finsler metrics under the weakly K\"ahler condition. In section 4, we give the examples of the unitary invariant complex Randers metrics which are weakly K\"ahler but not K\"ahler and Theorem \ref{mt} is proved.

\section{Preliminaries}

Let $M$ be an $n$-dimensional complex manifold. 
The canonical complex structure $J$ acts on  the complexified tangent bundle $T_{\mathbb{C}}M$ so that 
\[T_{\mathbb{C}}M=T^{1,0}M \oplus T^{0,1}M\]
where $T^{1,0}M$ is called the holomorphic tangent bundle. 

The set $\{z^1, \dots, z^n\}$ is the local complex coordinate on $M$, where $z^\alpha=x^\alpha+ix^{n+\alpha}$, $1\leq\alpha\leq n$,  and  
$\{x^1,\dots,x^n,x^{n+1},\dots,x^{2n}\}$ is the local real coordinate on $M$.  Let
\[\frac{\partial }{\partial z^\alpha}:=\frac{1}{2}(\frac{\partial}{\partial x^\alpha}-\sqrt{-1}\frac{\partial}{\partial x^{\alpha+n}} ), \quad \frac{\partial }{\partial \bar{z}^\alpha}:=\frac{1}{2}(\frac{\partial}{\partial x^\alpha}+\sqrt{-1}\frac{\partial}{\partial x^{\alpha+n}} ).\]
The set $\{\frac{\partial}{\partial z^1}, \dots, \frac{\partial}{\partial z^n} \}$ and $\{\frac{\partial}{\partial \bar{z}^1}, \dots, \frac{\partial}{\partial \bar{z}^n} \}$ are the local frames of $T^{1,0}M$ and  $T^{0,1}M$ respectively. Thus for any $(z, v)\in T^{1,0}M$,  the vector $v$ can  be expressed as $v=v^\alpha \frac{\partial}{\partial z^\alpha}$.

Similar to the Hermitian metric in complex geometry, the complex Finsler metric is defined as the follows.

\begin{definition}(See \cite{AP})
A complex Finsler metric $F$ on a complex manifold $M$ is a continuous function $F: T^{1,0}M\rightarrow [0, +\infty)$ satisfying
\begin{enumerate}
\item $G(z,v)=F^2(z,v)\geq 0$ and $G(z,v)=0$ if and only if $v=0$;
\item $G(z,\lambda v)=|\lambda|^2G(z,v)$ for all $(z,v)\in T^{0,1}M$ and $\lambda\in \mathbb{C}$;
\item $G$ is smooth on $T^{0,1}M$, we say $F$ is a smooth Finsler metric. 
\end{enumerate}
\end{definition}

\begin{definition} A complex Finsler metric $F$ is called strongly pseudo-convex if the Levi matrix 
   \[(G_{\alpha\bar{\beta}})=(\frac{\partial ^2 G}{\partial v^{\alpha}\partial \bar{v}^{\beta}})\]
is positive definite on $T^{0,1}M$. The tensor $G_{\alpha\bar{\beta}}dz^\alpha\otimes d\bar{z}^\beta$ is called the fundamental tensor. 
\end{definition}

There are several complex Finsler connections in complex Finsler geometry such as the Chern-Finsler connection and the complex Berwald connection. For our convenience, we use the Chern-Finsler connection to derive the holomorphic curvature. 

For a function $G(z,v)$ defined on $T^{0,1}M$, we denote 
\[G_{\alpha}:=\frac{\partial G}{\partial v^{\alpha}},\quad G_{;\alpha}:=\frac{\partial G}{\partial z^{\alpha}},\quad G_{\alpha;\bar{\beta}}:=\frac{\partial^2 G}{\partial v^\alpha\partial \bar{z}^{\beta}}.\]
Let
\[\frac{\delta}{\delta z^\alpha}:=\frac{\partial}{\partial z^\alpha}-N^{\beta}_{\alpha}\frac{\delta}{\delta v^{\beta}},\quad \delta v^\alpha:=dv^{\alpha}+N^{\alpha}_{\beta}dz^\beta,\quad N^\alpha_\beta:=G^{\alpha\bar{\gamma}}G_{\bar{\gamma};\beta},\]
then the tangent bundle of $T^{0,1}M$  splits into the horizontal and the vertical parts, i.e.: \[T_{\mathbb{C}}\tilde{M}=\mathcal{H}\oplus \bar{\mathcal{H}}\oplus \mathcal{V}\oplus\bar{\mathcal{V}}\] where $\tilde{M}=T^{0,1}M$, $\mathcal{H}=\text{span}\{\frac{\delta}{\delta z^{\alpha}}\}$ and $\mathcal{V}=\text{span}\{\frac{\partial}{\partial v^{\alpha}}\}$. The Chern-Finsler connection 1-forms are given by
\[\omega^\alpha_\beta=G^{\alpha\bar{\gamma}}\partial G_{\beta\bar{\gamma}}=\Gamma^\alpha_{\beta;\gamma}dz^\gamma+C^{\alpha}_{\beta\gamma}\delta v^{\gamma},\]
where the connection coefficients can be written as
\[\Gamma^\alpha_{\beta; \gamma}:=G^{\alpha\bar{\eta}}\frac{\delta G_{\beta\bar{\eta}}}{\delta z^\gamma}, \quad C^\alpha_{\beta\gamma}:=G^{\alpha\bar{\eta}}\frac{\partial G_{\beta\bar{\eta}}}{\partial v^\gamma}.\]
The curvature 2-forms $\Omega^\alpha_\beta$ are 
\[\Omega^\alpha_\beta:=\bar{\partial }\omega^\alpha_\beta=R^\alpha_{\beta;\gamma\bar{\eta}}dz^\gamma\wedge d\bar{z}^\eta+S^\alpha_{\beta\gamma;\bar{\eta}}\delta v^\gamma\wedge d\bar{z}^\eta+P^\alpha_{\beta\bar{\eta};\gamma}dz^\gamma\wedge \delta\bar{v}^\eta+Q^\alpha_{\beta\gamma\bar{\eta}}\delta v^\gamma\wedge \delta\bar{v}^\eta, \]
where 
\begin{eqnarray*}
R^\alpha_{\beta;\gamma\bar{\eta} }& = & -\delta_{\bar{\eta}}(\Gamma^\alpha_{\beta;\gamma})-C^\alpha_{\beta\mu}\delta_{\bar{\eta}}(N^\mu_\gamma),\\
 S^\alpha_{\beta\gamma;\bar{\eta}}& = & -\delta_{\bar{\eta}}(C^\alpha_{\beta\gamma}),\\
 P^\alpha_{\beta\bar{\eta};\gamma}& = & -\dot{\partial}_{\bar{\eta}}(\Gamma^\alpha_{\beta;\gamma})-C^\alpha_{\beta\mu}\dot{\partial}_{\bar{\eta}}(N^\mu_\gamma),\\
 Q^\alpha_{\beta\gamma\bar{\eta}}& = &-\dot{\partial}_{\bar{\eta}}(C^\alpha_{\beta\gamma}).
\end{eqnarray*}

\begin{definition} (See \cite{AP}) The holomorphic curvature of a complex Finsler metric $F$ is defined as
\[K_F(v):=\frac{1}{G^2}G_{\alpha}R^\alpha_{\beta;\gamma\bar{\eta} }v^{\beta}v^{\gamma}\bar{v}^{\eta}=-\frac{2}{G^2}G_{\alpha}\delta_{\bar{\gamma}}(N^\alpha_{\beta})v^\beta\bar{v}^{\gamma} .\]
\end{definition}

Similar to the Hermitian metric in complex geometry, the $(2,0)$-torsion of the Chern-Finsler connection is 
\[\theta(X,Y):=\nabla _{X}Y-\nabla_{Y}X-[X,Y],\ \  X, Y\in \mathcal{H}\oplus\mathcal{V}. \]
The K\"ahler condition means the above $(2,0)$ torsion is vanishing. 

\begin{definition} (See \cite{AP}) Let $F$ be a complex Finsler metric, and $\chi=v^\alpha\frac{\delta}{\delta z^{\alpha}}\in \Gamma (\mathcal{H})$ be the radial horizontal field. Then $F$ is
\begin{enumerate}
\item strongly K\"ahler, if $\theta(X,Y)=0$ for any $X, Y\in \mathcal{H}$;
\item K\"ahler, if $\theta(X,\chi)=0$  for any $X\in \mathcal{H}$;
\item weakly K\"ahler, if $\langle \theta(X,\chi), \chi\rangle=0$ for any $X\in \mathcal{H}$.
\end{enumerate}
\end{definition}

In a local coordinate, the torsion $\theta$ is given by 
\[\theta=(\Gamma^{\alpha}_{\beta;\gamma}dz^\beta\wedge dz^{\gamma}+C^\alpha_{\beta\gamma}\delta v^{\beta}\wedge dz^{\gamma})\otimes \frac{\delta}{\delta z^{\alpha}}.\]
Hence $F$ is a strongly K\"ahler metric if and only if $\Gamma^{\alpha}_{\beta;\gamma}=\Gamma^{\alpha}_{\gamma;\beta}$; $F$ is a K\"ahler metric if and only if $\Gamma^{\alpha}_{\beta;\gamma}v^{\gamma}=\Gamma^{\alpha}_{\gamma;\beta}v^{\gamma}$; $F$ is a weakly K\"ahler metric if and only if $G_{\alpha}\Gamma^{\alpha}_{\beta;\gamma}v^{\gamma}=G_{\alpha}\Gamma^{\alpha}_{\gamma;\beta}v^{\gamma}$.

\section{Unitary invariant Finsler metrics}

In complex Finsler geometry, there lack the canonical complex Finsler metrics comparing to Hermitian geometry.  In real Finsler geometry, the third author introduced the spherically symmetric Finsler metrics  which are invariant under any rotation in $\mathbb{R}^n$ \cite{Zhou}. Similarly, in \cite{Zhong}, Zhong introduced the unitary invariant complex Finsler metrics, which are invariant under any unitary action in $\mathbb{C}^n$, and gave an investigation on the complex  Chern-Finsler connection and holomorphic curvature of the metrics. In this section, we will recall some results in \cite{Zhong} and simplify some formulas by introducing $U$ and $W$.

\begin{definition} (See \cite{Zhong})
A complex Finsler metric $F$ on a domain $D \subset \mathbb{C}^n$ is called unitary invariant if $F$ satisfies
\[F(Az, Av)=F(z,v)\]
for any $(z,v)\in T^{1,0}M$ and $A\in U(n)$, where $U(n)$ are the unitary matrices over the complex number field $\mathbb{C}$.
\end{definition}

\begin{theorem} (See \cite{XZh1}) Let $F$ be a strongly pseudo-convex complex Finsler metric defined on a domain $D\subset \mathbb{C}^n$. The metric $F$ is unitary invariant if and only if 
there exists a smooth function $\phi(t,s): [0, +\infty)\times [0,+\infty)\rightarrow (0,+\infty)$ such that
\[F(z,v)=\sqrt{r\phi(t,s)}, \quad r:=|v|^2, t:=|z|^2, s:=\frac{|\langle z,v \rangle |^2}{r}\]
for every $(z,v)\in T^{1,0}D$.
\end{theorem}

The fundamental tensor of $F=\sqrt{r\phi(t,s)}$ can be written as 
\[G_{\alpha\bar{\beta}}=(\phi-s\phi_s)\delta_{\alpha\bar{\beta}}+r\phi_{ss}s_{\alpha}s_{\bar\beta}+\phi_s \bar{z}^\alpha z^\beta.\]
The determinant of the matrix $(G_{\alpha,\bar{\beta}})$ is 
\[\det(G_{\alpha,\bar{\beta}})=\{(\phi-s\phi_s)[\phi+(t-s)\phi_s]+s(t-s)\phi\phi_{ss}\}(\phi-s\phi_s)^{n-2}.\] 

\begin{proposition} (See \cite{Zhong}) The unitary invariant complex Finsler metric $F=\sqrt{r\phi(t,s)}$ is strongly pseudo-convex on a domain $D\subset \mathbb{C}^n$ if and only if 
\[\phi-s\phi_s>0, \quad (\phi-s\phi_s)[\phi+(t-s)\phi_s]+s(t-s)\phi\phi_{ss}>0.\]
\end{proposition}

In \cite{Zhong}, Zhong showed that a strongly pseudo-convex complex Finsler metric $F=\sqrt{r\phi(t,s)}$ defined on a domain $D\subset \mathbb{C}^n$ is a K\"ahler Finsler metric if and only if $\phi(t,s)=a(t)+a^{\prime}(t)s$, where $a(t)$ is a positive smooth function satisfying $a(t)+ta^{\prime}(t)>0$. For the weakly K\"ahler case, Zhong obtained the following 
result. 
\begin{theorem} (See \cite{Zhong})
 The unitary invariant complex Finsler metric $F=\sqrt{r\phi(t,s)}$ defined on a domain $D\subset \mathbb{C}^n$ is weakly K\"ahler if and only if 
 \begin{equation}\label{Kaeq}(\phi-s\phi_s)[\phi+(t-s)\phi_s][\phi_s-\phi_t+s(\phi_{st}+\phi_{ss})]+s(t-s)\phi_{ss}[\phi(\phi_s-\phi_t)+s\phi_s(\phi_t+\phi_s)]=0.\end{equation}
\end{theorem}

Actually, we can simplify the equation (\ref{Kaeq}) and yield the following theorem.

\begin{theorem} \label{wkeq}
An unitary invariant complex Finsler metric $F=\sqrt{r\phi(t,s)}$ defined on a domain $D\subset \mathbb{C}^n$ is weakly K\"ahler if and only if
\begin{equation}\label{nKaeq} sU(U-t)W_s-s(U-t)U_sW-2(U-s)U_s=0\end{equation}
where $U=\frac{s\phi+s(t-s)\phi_s}{\phi}$ and $W=\frac{\phi_t+\phi_s}{\phi}$.
\end{theorem}
\begin{proof}
Let 
$$U:=\frac{s\phi+s(t-s)\phi_s}{\phi}, \quad W:=\frac{\phi_t+\phi_s}{\phi}.$$
Then we have
 \begin{equation} \phi_s=\frac{U-s}{s(t-s)}\phi, \quad \phi_t=(W-\frac{U-s}{s(t-s)})\phi.\label{pspt} \end{equation}
It is easy to see
\[\phi-s\phi_s=\frac{t-U}{t-s}\phi, \quad \phi+(t-s)\phi_s=\frac{U}{s}\phi.\]
Obviously, $\phi_s-\phi_t=-W\phi+2\frac{U-s}{s(t-s)}\phi$ and $\phi_t+\phi_s=W\phi$, so that 
\[\phi(\phi_s-\phi_t)+s\phi_s(\phi_t+\phi_s)=\frac{sW(U-t)+2(U-s)}{s(t-s)}\phi^2.\]
By using $\phi_{st}+\phi_{ss}=(\phi_t+\phi_s)_s=W_s\phi+W\phi_s$, we have
\[\phi_s-\phi_t+s(\phi_{st}+\phi_{ss})=s W_s\phi+\frac{U-t}{t-s}W\phi+\frac{2(U-s)}{s(t-s)}\phi.\]
Moreover, from the first equation of (\ref{pspt}), we compute \[\phi_{ss}=\frac{s(t-s)U_s+U(U-t)}{s^2(t-s)^2}\phi.\]
Substituting  the above equalities into the weakly K\"ahler equation (\ref{Kaeq}) implies the equation ({\ref{nKaeq}}).   

\end{proof}

\begin{lemma} \label{lem0}
$U$ and $W$ in Theorem \ref{wkeq} satisfy the following integral  equation
\[s(U_t+U_s)=s^2(t-s)W_s+U.\]
\end{lemma}
\begin{proof}
From the formula of $U$ and $W$, we can see that
\[\phi_s=\frac{U-s}{s(t-s)}\phi, \quad \phi_t=(W-\frac{U-s}{s(t-s)})\phi.\]
Then $\phi_{st}=\phi_{ts}$ implies the result. 
\end{proof}

\begin{theorem} \label{hct1}
Let  $F=\sqrt{r\phi(t,s)}$ be the unitary invariant complex Finsler metric defined on a domain $D\subset \mathbb{C}^n$ and $K_F$ be the holomorphic curvature. Then in terms of $\phi$, $K_F$ is
\begin{eqnarray*}
K_F(v)&=&-\frac{2}{\phi}\Big\{[s(\frac{\partial k_2}{\partial t}+\frac{\partial k_2}{\partial s})+k_2]+\frac{s\phi+s(t-s)\phi_s}{\phi}[s(\frac{\partial k_3}{\partial t}+\frac{\partial k_3}{\partial s})+2k_3]\Big\},
\end{eqnarray*}
where 
\begin{eqnarray*}
k_1&:=&(\phi-s\phi_s)[\phi+(t-s)\phi_s]+s(t-s)\phi\phi_{ss},\\
k_2&:=&\frac{1}{k_1}\{[\phi+(t-s)\phi_s+s(t-s)\phi_{ss}](\phi_t+\phi_s)-s[\phi+(t-s)\phi_s](\phi_{st}+\phi_{ss})\},\\
k_3&:=&\frac{1}{k_1}[\phi(\phi_{st}+\phi_{ss})-\phi_s(\phi_t+\phi_s)].
\end{eqnarray*}
\end{theorem}
\begin{proof}
According to the definition of the holomorphic curvature, we have
\begin{eqnarray} \label{hceq}
K_F&=&-\frac{2}{G^2}G_{\alpha}\delta_{\bar{\gamma}}(N^\alpha_{\beta})v^\beta\bar{v}^{\gamma} \nonumber\\
&=&-\frac{2}{G^2}G_{\gamma}\delta_{\bar{\nu}}(2\mathbb{G}^\gamma)\bar{v}^{\nu} \nonumber \\
   &=&-\frac{2}{G^2}G_{\gamma}\frac{\partial }{\partial \bar{z}^{\nu}}(2\mathbb{G}^\gamma)\bar{v}^{\nu}+\frac{2}{G^2}G_{\gamma}\bar{N}^\alpha_\nu\frac{\partial }{\partial \bar{v}^\alpha}(2\mathbb{G}^\gamma)\bar{v}^{\nu}
\end{eqnarray}
where $\mathbb{G}^\gamma:=\frac{1}{2}N^\gamma_{\beta}v^\beta$.
For the unitary invariant complex Finsler metric $F=\sqrt{r\phi(t,s)}$, the complex spray coefficients are \cite{Zhong}
 $$2\mathbb{G}^\gamma=N^\gamma_\beta v^\beta=k_2\overline{\langle z, v\rangle}v^\gamma+k_3\overline{\langle z, v\rangle}^2z^\gamma.$$
 Therefore one needs to calculate
 \begin{eqnarray*}
 \frac{\partial}{\partial {\bar z^\nu}}(2\mathbb{G}^\gamma)&=&\frac{\partial k_2}{\partial t}\overline{\langle z, v\rangle}z^\nu v^\gamma+\frac{\partial k_2}{\partial s}sv^\nu v^\gamma+k_2v^\nu v^\gamma \\
 &&+\frac{\partial k_3}{\partial t}\overline{\langle z, v\rangle}^2z^\nu z^\gamma +\frac{\partial k_3}{\partial s}s\overline{\langle z, v\rangle}v^\nu z^\gamma    +2k_3\overline{\langle z, v\rangle}v^\nu z^\gamma
 \end{eqnarray*}
 and 
 \begin{eqnarray*}
 \frac{\partial}{\partial {\bar z^\nu}}(2\mathbb{G}^\gamma)\bar{v}^\nu=r[s(\frac{\partial k_2}{\partial t}+\frac{\partial k_2}{\partial s})+k_2]v^\gamma+r[s(\frac{\partial k_3}{\partial t}+\frac{\partial k_3}{\partial s})+2k_3]\overline{\langle z, v\rangle}z^\gamma.
 \end{eqnarray*}
 Since $G=F^2=r\phi(t,s)$, we have $G_\gamma=\bar{v}^\gamma\phi+r\phi_s s_\gamma$, where $s_\gamma:=\frac{\partial s}{\partial v^\gamma}=-r^{-2}\bar{v}^\gamma |\langle z, v\rangle|^2+r^{-1}\langle z, v\rangle\bar{z}^\gamma$. Obviously, $s_\gamma v^\gamma=0$ and $s_\gamma z^\gamma=r^{-1}\langle z, v\rangle (t-s)$. So we have
\begin{eqnarray}\label{hc1}-\frac{2}{G^2}G_{\gamma}\frac{\partial }{\partial \bar{z}^{\nu}}(2\mathbb{G}^\gamma)\bar{v}^{\nu}&=&-\frac{2}{r^2\phi^2}(\bar{v}^\gamma\phi+r\phi_s s_\gamma)\Big\{r[s(\frac{\partial k_2}{\partial t}+\frac{\partial k_2}{\partial s})+k_2]v^\gamma   \nonumber\\ 
&&+r[s(\frac{\partial k_3}{\partial t}+\frac{\partial k_3}{\partial s})+2k_3]\overline{\langle z, v\rangle}z^\gamma\Big\}  \nonumber \\
&=&-\frac{2}{\phi}\Big\{[s(\frac{\partial k_2}{\partial t}+\frac{\partial k_2}{\partial s})+k_2] \nonumber\\
        &&+\frac{s\phi+s(t-s)\phi_s}{\phi}[s(\frac{\partial k_3}{\partial t}+\frac{\partial k_3}{\partial s})+2k_3]\Big\}.
\end{eqnarray}

On the other hand, it is necessary to compute
\begin{eqnarray*}
\frac{\partial }{\partial \bar{v}^\alpha}(2\mathbb{G}^\gamma)&=&\frac{\partial }{\partial \bar{v}^\alpha}(k_2\overline{\langle z, v\rangle}v^\gamma+k_3\overline{\langle z, v\rangle}^2z^\gamma)\\
&=&\frac{\partial k_2}{\partial s}s_{\bar{\alpha}}\overline{\langle z, v\rangle}v^\gamma+\frac{\partial k_3}{\partial s}s_{\bar{\alpha}}\overline{\langle z, v\rangle}^2z^\gamma
\end{eqnarray*}
and 
\begin{eqnarray*}
\bar{N}^\alpha_\nu\frac{\partial }{\partial \bar{v}^\alpha}(2\mathbb{G}^\gamma)\bar{v}^{\nu}&=&2\bar{\mathbb{G}}^\alpha\frac{\partial }{\partial \bar{v}^\alpha}(2\mathbb{G}^\gamma)\\
&=&2(k_2\langle  z,v\rangle \bar{v}^\alpha+k_3\langle  z,v\rangle^2\bar{z}^\alpha)(\frac{\partial k_2}{\partial s}s_{\bar{\alpha}}\overline{\langle z, v\rangle}v^\gamma+\frac{\partial k_3}{\partial s}s_{\bar{\alpha}}\overline{\langle z, v\rangle}^2z^\gamma)\\
&=&rs^2(t-s)k_3[\frac{\partial k_2}{\partial s}v^\gamma+\frac{\partial k_3}{\partial s}\overline{\langle z, v\rangle}z^\gamma].
\end{eqnarray*}
Thus we obtain
\begin{eqnarray} \label{hc2}
\frac{2}{G^2}G_{\gamma}\bar{N}^\alpha_\nu\frac{\partial }{\partial \bar{v}^\alpha}(2\mathbb{G}^\gamma)\bar{v}^{\nu}&=&\frac{2}{r^2\phi^2}(\bar{v}^\gamma\phi+r\phi_s s_\gamma)rs^2(t-s)k_3[\frac{\partial k_2}{\partial s}v^\gamma+\frac{\partial k_3}{\partial s}\overline{\langle z, v\rangle}z^\gamma] \nonumber\\
&=&\frac{2}{\phi}s^2(t-s)k_3[\frac{\partial k_2}{\partial s}+\frac{s\phi+s(t-s)\phi_s}{\phi}\frac{\partial k_3}{\partial s}].
\end{eqnarray}
Plugging (\ref{hc1}) and (\ref{hc2}) into (\ref{hceq}) will yield 
\begin{eqnarray*}
K_F(v)&=&-\frac{2}{\phi}\Big\{[s(\frac{\partial k_2}{\partial t}+\frac{\partial k_2}{\partial s})+k_2]+\frac{s\phi+s(t-s)\phi_s}{\phi}[s(\frac{\partial k_3}{\partial t}+\frac{\partial k_3}{\partial s})+2k_3]\\
&&-s^2(t-s)k_3[\frac{\partial k_2}{\partial s}+\frac{s\phi+s(t-s)\phi_s}{\phi}\frac{\partial k_3}{\partial s}]\Big\}.
\end{eqnarray*}
By a direct computation, we notice that
\[\frac{\partial k_2}{\partial s}+\frac{s\phi+s(t-s)\phi_s}{\phi}\frac{\partial k_3}{\partial s}=0.\]
So the the formula of $K_F(v)$ holds. 

\end{proof}

\begin{remark} In \cite{WXZ}, Wang, Xia and Zhong also obtained the formula of the holomorphic curvature of the unitary invariant complex Finsler metrics. 

\end{remark}

When the metric is weakly K\"ahler, the holomorphic curvature has a nice formula in terms of $U$ and $W$.
 
\begin{theorem} Let $F$ be an unitary invariant metric defined on a domain $D\subset \mathbb{C}^n$. If $F=\sqrt{r\phi(t,s)}$ is a weakly K\"ahler Finsler metric, the holomorphic curvature $K_F$ is given by
\[K_F=-\frac{2}{\phi}\{s(W_t+W_s)-s^2(t-s)\frac{W_s^2}{U_s}+W\}\]
where $U=\frac{s\phi+s(t-s)\phi_s}{\phi}$ and $W=\frac{\phi_t+\phi_s}{\phi}$.
\begin{proof}
From 
$$U=\frac{s\phi+s(t-s)\phi_s}{\phi}, \quad W=\frac{\phi_t+\phi_s}{\phi},$$
one can obtain 
\[\phi_s=\frac{U-s}{s(t-s)}\phi, \quad \phi_t=(W-\frac{U-s}{s(t-s)})\phi.\]
Substituting above equalities into the formulas of $k_1$, $k_2$ and $k_3$ in Theorem \ref{hct1} derives
\[k_1=U_s\phi^2,\quad  k_2=\frac{WU_s-UW_s}{U_s},\quad k_3=\frac{W_s}{U_s}.\]

When imposing the weakly K\"ahler condition i.e.
   \[sU(U-t)W_s-s(U-t)U_sW-2(U-s)U_s=0,\]
   one can simplify $k_2$ and $k_3$ to get
   \[k_2=-\frac{2(U-s)}{s(U-t)},\quad k_3=\frac{W-k_2}{U}=\frac{W}{U}+\frac{2(U-s)}{sU(U-t)}.\] 
 Thus
 \[\frac{\partial k_2}{\partial t}=\frac{2(t-s)U_t-2(U-s)}{s(U-t)^2},\quad \frac{\partial k_2}{\partial s}=\frac{2s(t-s)U_s+2U(U-t)}{s^2(U-t)^2},\]
 \[\frac{\partial k_3}{\partial t}=(\frac{W}{U})_t+\frac{U_t}{U^2}k_2-\frac{1}{U}\frac{\partial k_2}{\partial t},\quad \frac{\partial k_3}{\partial s}=(\frac{W}{U})_s+\frac{U_s}{U^2}k_2-\frac{1}{U}\frac{\partial k_2}{\partial s}.\]

 According to the formula of the holomorphic curvature $K_F$ in Theorem \ref{hct1}, the above formulas involving $k_2$ and $k_3$ will imply
 \[K_F=-\frac{2}{\phi}\Big\{s(W_t+W_s)-\frac{sW(U-t)+2(U-s)}{U(U-t)}(U_t+U_s)+\frac{2sW(U-t)+2(U-s)}{s(U-t)}\Big\}.\]
 From the integral equation of $U$ and $W$ in Lemma \ref{lem0}, we have
   \[U_t+U_s=s(t-s)W_s+\frac{U}{s}.\]
Therefore the holomorphic curvature $K_F$ can be simplified as
\[K_F=-\frac{2}{\phi}\Big\{s(W_t+W_s)-\frac{s^2(t-s)WW_s}{U}-s(t-s)\frac{2(U-s)}{U(U-t)}W_s+W\Big\}.\]
Notice that the weakly K\"ahler condition tells us
\[\frac{2(U-s)}{U(U-t)}=\frac{sW_s}{U_s}-\frac{sW}{U}.\]
It leads to
\[K_F=-\frac{2}{\phi}\{s(W_t+W_s)-s^2(t-s)\frac{W_s^2}{U_s}+W\},\]
as we want in the Theorem.

\end{proof}

\end{theorem}

\section{Weakly K\"ahler unitary invariant complex Finsler metrics with Randers type and an uniformization theorem}

In 2009, Aldea and Munteanu initiated the study of the complex Randers motivated by the notion of the real Randers metrics \cite{AM}. Later, Chen and Shen gave the formula of the holomorphic curvature for complex Randers metrics and proved some rigidity theorems on complex Randers metrics \cite{CS1}.

\begin{definition}
  A complex Finsler metric $F$ on a complex manifold is called a complex Randers metric if $F$ can be written as $F=\alpha+|\beta|$, where $\alpha$ is a Hermitian metric and $\beta$ is a $(1,0)$-form, i.e. $\alpha=\sqrt{a_{i \bar{j}}(z) v^{i} \bar{v}^{j}}$, $\beta=b_{i}(z) v^{i}$.
\end{definition}

\begin{remark}(1) The complex Randers metric is not smooth along the directions $v=v^i\frac{\partial}{\partial z^i}$ with $b_iv^i=0$.

(2) According to above definition,  it is easy to see that an unitary invariant Finsler metric $F=\sqrt{r\phi(t,s)}$ is complex Randes metric if and only if 
\[\phi=\big(\sqrt{f(t)+g(t)s}+\sqrt{h(t)s} \big)^2\]
 where $f(t), g(t), h(t)$ are smooth functions and $f(t)>0, h(t)\geq0$.
\end{remark}

\begin{theorem} \label{wkr}An unitary invariant Randers metric $F=\sqrt{r\phi(t,s)}$ defined on a domain $D \subset \mathbb{C}^n$ is a weakly K\"ahler Finsler metric if and only if
\[\phi=\Big (\sqrt{f(t)+\frac{tf^{'}(t)-f(t)}{2t}s}+\sqrt{\frac{tf^{'}(t)+f(t)}{2t}s}\Big )^2\]
or
\[\phi=f(t)+f^\prime(t)s\]
where $f(t)$ is a positive smooth function.   

\end{theorem}
\begin{proof} 
The sufficiency is obvious by a direct computation. 

Now we prove the necessity.
Since $F=\sqrt{r\phi(t,s)}$ is a weakly K\"ahler complex Randers metric,   $\phi$ can be written as 
  $$\phi=\big(\sqrt{f(t)+g(t)s}+\sqrt{h(t)s} \big)^2$$
 where $f(t), g(t), h(t)$ are smooth functions and $f(t)>0, h(t)\geq 0$.
 Therefore
 \[U=\frac{s\phi+s(t-s)\phi_s}{\phi}=\frac{s(f+tg)+t\sqrt{s(f+gs)h}}{\sqrt{f+gs}(\sqrt{f+gs}+\sqrt{hs})},\]
 \[W=\frac{\phi_t+\phi_s}{\phi}=\frac{(h+sh^\prime)\sqrt{f+gs}+(f^\prime+g+sg^\prime)\sqrt{sh}}{(\sqrt{f+gs}+\sqrt{hs})\sqrt{s(f+gs)h}}.\]
Substituting above equalities into the equation of weakly K\"ahler condition (\ref{nKaeq}), one will get a quite complex equation as following
\[\{A_0(t)+A_1(t)\sqrt{f+gs}\sqrt{s}+A_2(t)(\sqrt{s})^2+A_3(t)\sqrt{f+gs}(\sqrt{s})^3+A_4(t)(\sqrt{s})^4\}f(t)=0\]
where \[A_0(t)=tfh(tf^\prime-f-2tg).\]
Notice that $s=\frac{|\langle z,v \rangle|^2 }{|v|^2}$ can be an arbitrary real number,  thus
   \[A_0(t)=A_1(t)=A_2(t)=A_3(t)=A_4(t)=0.\]
From $A_0(t)=0$, we know that one of the following three equations holds:
\[h(t)=0,\quad f(t)=0\quad \text{or} \quad g(t)=\frac{tf^\prime(t)-f(t)}{2t}.\]

If $h(t)=0$, then the complex Randers metric $F$  reduces to the Hermitian metric and then must be K\"ahler. It can imply 
\[\phi=f(t)+f^\prime(t)s.\]

Since $F$ is a complex Randers metric, $f(t)\neq 0$.

If $ g(t)=\frac{tf^\prime(t)-f(t)}{2t}$, substituting it into the equation of weakly K\"ahler condition again and rearranging all the terms,  one can obtain
\[\{B_1(t)\sqrt{f+gs}\sqrt{s}+B_2(t)(\sqrt{s})^2+B_3(t)\sqrt{f+gs}(\sqrt{s})^3+B_4(t)(\sqrt{s})^4\}f(t)=0\]
where 
\[B_1(t)=t\sqrt{h}(tf^\prime+f)(tf^\prime+f-2th).\]
Therefore 
\[B_1(t)=B_2(t)=B_3(t)=B_4(t)=0.\]
So $B_1(t)=0$ deduces that one of the following three equations holds
\[h(t)=0,\quad tf^\prime(t)+f(t)=0\quad \text{or} \quad h(t)=\frac{tf^\prime(t)+f(t)}{2t}.\]

For the case of $h(t)=0$, we have discussed previously. If $tf^\prime(t)+f(t)=0$ i.e. $f(t)=\frac{c}{t}$ with $c$ a constant, then 
\[\phi=\big(\sqrt{\frac{c(t-s)}{t^2}}+\sqrt{h(t)s}\big)^2\]
which means the Hermitian part $\sqrt{\frac{c(t-s)}{t^2}}$ is not positive defined.  Therefore we have
\[h(t)=\frac{tf^\prime(t)+f(t)}{2t}.\]
\end{proof}

\begin{remark} It is easy to check that if $f(t)>0$ and $f^{\prime}(t)>0$, the weakly K\"ahler complex Randers metric in above theorem is strongly pseudo-convex.
\end{remark}

Furthermore, we can prove the following uniformization theorem in complex Finsler geometry. 

\begin{theorem} 
Let $F$ be an unitary invariant complex Randers metric defined on a domain $D\subset \mathbb{C}^n$. Assume $F$ is not a Hermitian metric. If $F=\sqrt{r\phi(t,s)}$ is a weakly K\"ahler Finsler metric and has a constant holomorphic curvature $k$ if and only if 
\begin{enumerate}
\item $k=4$, $\phi=(\sqrt{\frac{t}{c^2+t^2}-\frac{t^2}{(c^2+t^2)^2}s}+\sqrt{\frac{c^2}{(c^2+t^2)^2}s})^2$ defined on $D=\mathbb{C}^n\setminus \{0\}$;
\item $k=0$, $\phi=c(\sqrt{t}+\sqrt{s})^2$ defined on $D=\mathbb{C}^n$;
\item $k=-4$, $\phi=\sqrt{\frac{t}{c^2-t^2}+\frac{t^2}{(c^2-t^2)^2}s}+\sqrt{\frac{c^2}{(c^2-t^2)^2}s}$ defined on $D=\{z: |z|<\sqrt{c}\}$
\end{enumerate}
where $c$ is a positive constant. 
\end{theorem}
 \begin{proof}
 The sufficiency is obvious by a direct computation. We only need to prove the necessity. 
 By Theorem \ref{wkr}, the complex Randers metric $F$ can be written as
 \[\phi=\Big (\sqrt{f(t)+\frac{tf^{'}(t)-f(t)}{2t}s}+\sqrt{\frac{tf^{'}(t)+f(t)}{2t}s}\Big )^2.\]
 Then we have
 \[U=\frac{s(f+tg)+t\sqrt{s(f+gs)h}}{\sqrt{f+gs}(\sqrt{f+gs}+\sqrt{hs})},\quad W=\frac{(h+sh^\prime)\sqrt{f+gs}+(f^\prime+g+sg^\prime)\sqrt{sh}}{(\sqrt{f+gs}+\sqrt{hs})\sqrt{s(f+gs)h}}\]
 where $g=\frac{tf^\prime(t)-f(t)}{2t}$ and $h=\frac{tf^\prime(t)+f(t)}{2t}$.
 
 If the holomorphic curvature $k=0$, one can get
 \[K_F=-\frac{2}{\phi}\{s(W_t+W_s)-s^2(t-s)\frac{W_s^2}{U_s}+W\}=0.\]
 Plugging $U$ and $W$ into the above equation will yield
 \[A_0(t)\sqrt{f+gs}+A_1(t)\sqrt{s}+\dots+A_7(t)(\sqrt{s})^7=0\]
 where $A_0(t)=t^4\sqrt{h}f^2(tf^\prime+f)(tf^\prime-f)$. Hence $A_0(t)=0$. Since $F$ is not a Hermitian metric, we know that
 \[tf^\prime-f=0.\]
 It implies that $f(t)=ct$ and $\phi=c(\sqrt{t}+\sqrt{s})^2$.
 
 If the holomorphic curvature $k=4$, one can get
 \[K_F=-\frac{2}{\phi}\{s(W_t+W_s)-s^2(t-s)\frac{W_s^2}{U_s}+W\}=4.\]
 Plugging $\phi$, $U$ and $W$ into the above equation will yield
 \[B_0(t)\sqrt{f+gs}+B_1(t)\sqrt{s}+\dots+B_7(t)(\sqrt{s})^7=0\]
 where $B_0(t)=t^4\sqrt{h}f^2(tf^\prime+f)(tf^\prime+2tf^2-f)$. Hence $B_0(t)=0$. From it, we know that
 \[tf^\prime+2tf^2-f=0.\]
 It implies that $f(t)=\frac{t}{c^2+t^2}$ and $\phi=(\sqrt{\frac{t}{c^2+t^2}-\frac{t^2}{(c^2+t^2)^2}s}+\sqrt{\frac{c^2}{(c^2+t^2)^2}s})^2$.
 
 If the holomorphic curvature $k=-4$, one can get
 \[K_F=-\frac{2}{\phi}\{s(W_t+W_s)-s^2(t-s)\frac{W_s^2}{U_s}+W\}=-4.\]
 Plugging $\phi$, $U$ and $W$ into the above equation will yield
 \[C_0(t)\sqrt{f+gs}+C_1(t)\sqrt{s}+\dots+C_7(t)(\sqrt{s})^7=0\]
 where $C_0(t)=t^4\sqrt{h}f^2(tf^\prime+f)(tf^\prime-2tf^2-f)$. Hence $C_0(t)=0$. From it, we know that
 \[tf^\prime-2tf^2-f=0.\]
 It implies that $f(t)=\frac{t}{c^2-t^2}$ and $\phi=\sqrt{\frac{t}{c^2-t^2}+\frac{t^2}{(c^2-t^2)^2}s}+\sqrt{\frac{c^2}{(c^2-t^2)^2}s}$.

 \end{proof}

\end{document}